\documentclass[12pt]{article}
\usepackage{mathrsfs}

\usepackage{t1enc}
\usepackage[latin1]{inputenc}
\usepackage[english]{babel}
\usepackage{amsmath,amsthm}
\usepackage{amsfonts}
\usepackage{latexsym}
\usepackage{graphicx}
\usepackage[natural]{xcolor}
\theoremstyle{plain}
\newtheorem{theorem}{Theorem}
\newtheorem{proposition}{Proposition}[section]
\newtheorem{definition}{Definition}[section]
\newtheorem{lemma}{Lemma}[section]

\newtheorem{claim}{Claim}
\newtheorem{case}{Case}

 \renewenvironment{proof}{
 \noindent{\bf Proof.}\rm} {\mbox{}\hfill\rule{0.5em}{0.809em}\par}

\usepackage{caption,subcaption}

\begin{document}
\date{}
\title{{Star-critical Gallai-Ramsey numbers of graphs} \footnote{This work is supported
by the Science and Technology Program of Guangzhou, China(No.202002030183) and by the Natural Science Foundation of Qinghai, China (No.2020-ZJ-924).
 Correspondence should be addressed to Yan
Liu(e-mail:liuyan@scnu.edu.cn)}
}
\author{\small Xueli Su, Yan Liu\\
\small School of Mathematical Sciences, South China Normal University,\\
 \small Guangzhou, 510631, P.R. China}

\maketitle

\setcounter{theorem}{0}

\begin{abstract}
The Gallai-Ramsey number $gr_{k}(K_{3}: H_{1}, H_{2}, \cdots, H_{k})$ is the smallest integer $n$ such that every $k$-edge-colored $K_{n}$ contains either a rainbow $K_3$ or a monochromatic $H_{i}$ in color $i$ for some $i\in [k]$.
We find the largest star that can be removed from $K_n$ such that the underlying graph is still forced to have a rainbow $K_3$ or a monochromatic $H_{i}$ in color $i$ for some $i\in [k]$.
Thus, we define the star-critical Gallai-Ramsey number $gr_{k}^{*}(K_3: H_{1}, H_{2}, \cdots, H_{k})$ as the smallest integer $s$ such that every $k$-edge-colored $K_{n}-K_{1, n-1-s}$ contains either a rainbow $K_3$ or a monochromatic $H_{i}$ in color $i$ for some $i\in [k]$.
When $H=H_{1}=\cdots=H_{k}$, we simply denote $gr_{k}^{*}(K_{3}: H_{1}, H_{2}, \cdots, H_{k})$ by $gr_{k}^{*}(K_{3}: H)$.
We determine the star-critical Gallai-Ramsey numbers for complete graphs and some small graphs. Furthermore, we show that $gr_{k}^{*}(K_3: H)$ is exponential in $k$ if $H$ is not bipartite, linear in $k$ if $H$ is bipartite but not a star and constant (not depending on $k$) if $H$ is a star.

\noindent {\bf Key words:} Gallai-Ramsey number, critical graph, star-critical.
\end{abstract}

\vspace{4mm}
\section{Introduction}

All graphs considered in this paper are finite, simple and undirected. For a graph $G$, we use $|G|$ to denote the number of vertices of $G$, say the \emph{order} of $G$.
The complete graph of order $n$ is denoted by $K_{n}$ and the star graph of order $n$ is denoted by $K_{1, n-1}$.
For a subset $S$ of $V(G)$, let $G[S]$ be the subgraph of $G$ induced by $S$.
For two disjoint subsets $A$ and $B$ of $V(G)$, $E(A, B)=\{ab\in E(G) ~|~ a\in A, b\in B\}$.
Let $G$ be a graph and $H$ a subgraph of $G$. The graph obtained from $G$ by deleting all edges of $H$ is denoted by $G-H$.
$K_{n-1}\sqcup K_{1, s}$ is a graph obtained from $K_{n-1}$ by adding a new vertex $v$ and adding $s$ edges which join $v$ to $s$ vertices of $K_{n-1}$.
For any positive integer $k$, we write $[k]$ for the set $\{1, 2, \cdots, k\}$.
An edge-colored graph is called \emph{monochromatic} if all edges are colored by the same color and \emph{rainbow} if no two edges are colored by the same color.
A \emph{blow-up} of an edge-colored graph $G$ on a graph $H$ is a new graph obtained from $G$ by replacing each vertex of $G$ with $H$ and replacing each edge $e$ of $G$ with a monochromatic complete bipartite graph $(V(H), V(H))$ in the same color with $e$.

The \emph{Ramsey number} $R(G, H)$ is the smallest integer $n$ such that every red-blue edge-colored $K_n$ contains either a red $G$ or a blue $H$.
When $G=H$, we simply denote $R(G, H)$ by $R_{2}(G)$.
For more information on Ramsey number, we refer the readers to a dynamic survey on Ramsey number in \cite{R}.
The definition of Ramsey number implies that there exists a \emph{critical graph}, that is, a red-blue edge-colored $K_{n-1}$ contains neither a red $G$ nor a blue $H$.
It is significant to find a smallest integer $s$ such that for any critical graph $K_{n-1}$, every red-blue edge-colored $K_{n-1}\sqcup K_{1, s}$ contains either a red $G$ or a blue $H$.
To study this, Hook and Isaak \cite{JG} introduced the definition of the \emph{star-critical Ramsey number} $r_{*}(G, H)$. $r_{*}(G, H)$ is the smallest integer $s$ such that every red-blue edge-colored $K_{n-1}\sqcup K_{1, s}$ contains either a red $G$ or a blue $H$.
Then it is clear that $r_{*}(G, H)$ is the smallest integer $s$ such that every red-blue edge-colored $K_{n-1}\sqcup K_{1, s}$ contains either a red $G$ or a blue $H$ for any critical graph $K_{n-1}$.
When $G=H$, we simply denote $r_{*}(G, H)$ by $r_{*}(G)$.
Star-critical Ramsey number of a graph is closely related to its upper size Ramsey number and lower size Ramsey number\cite{JG}.
In \cite{ZBC}, Zhang et al. defined the \emph{Ramsey-full} graph pair $(G, H)$.
A graph pair $(G, H)$ is called Ramsey-full if there exists a red-blue edge-colored graph $K_{n}-e$ contains neither a red $G$ nor a blue $H$, where $n=R(G, H)$.
So $(G, H)$ is Ramsey-full if and only if $r_{*}(G, H)=n-1$.
If $(G, H)$ is Ramsey-full and $G=H$, then we say that $H$ is Ramsey-full.
Hook and Isaak \cite{J} proved that the complete graph pair $(K_m, K_n)$ is Ramsey-full.

Given a positive integer $k$ and graphs $H_{1}, H_{2}, \cdots, H_{k}$, the \emph{Gallai-Ramsey number} $gr_{k}(K_{3}: H_{1}, H_{2}, \cdots, H_{k})$ is the smallest integer $n$ such that every $k$-edge-colored $K_{n}$ contains either a rainbow $K_3$ or a monochromatic $H_{i}$ in color $i$ for some $i\in [k]$.
Clearly, $gr_{2}(K_{3}: H_{1}, H_{2})=R(H_{1}, H_{2})$.
When $H=H_{1}=\cdots=H_{k}$, we simply denote $gr_{k}(K_{3}: H_{1}, H_{2}, \cdots, H_{k})$ by $gr_{k}(K_{3}: H)$.
More information on Gallai-Ramsey number can be found in \cite{FMC, CP}.
Let $n=gr_{k}(K_3: H_{1}, H_{2}, \cdots, H_{k})$.
The definition of Gallai-Ramsey number implies that there exists a \emph{critical graph}, that is, a $k$-edge-colored $K_{n-1}$ contains neither a rainbow $K_3$ nor a monochromatic $H_{i}$ for any $i\in [k]$.
In this paper, we define the \emph{star-critical Gallai-Ramsey number} $gr_{k}^{*}(K_3: H_{1}, H_{2}, \cdots, H_{k})$ to be the smallest integer $s$ such that every $k$-edge-colored graph $K_{n-1}\sqcup K_{1, s}$ contains either a rainbow $K_3$ or a monochromatic $H_{i}$ in color $i$ for some $i\in [k]$.
Then it is clear that $gr_{k}^{*}(K_3: H_{1}, H_{2}, \cdots, H_{k})$ is the the smallest integer $s$ such that for any critical graph $K_{n-1}$, every $k$-edge-colored graph $K_{n-1}\sqcup K_{1, s}$ contains either a rainbow $K_3$ or a monochromatic $H_{i}$ in color $i$ for some $i\in [k]$.
Clearly, $gr_{k}^{*}(K_3: H_{1}, H_{2}, \cdots, H_{k})\leq n-1$ and $gr_{2}^{*}(K_{3}: H_{1}, H_{2})=r_{*}(H_{1}, H_{2})$. 
When $H=H_{1}=\cdots=H_{k}$, we simply denote $gr_{k}^{*}(K_{3}: H_{1}, H_{2}, \cdots, H_{k})$ by $gr_{k}^{*}(K_{3}: H)$.
$(H_{1}, H_{2}, \cdots, H_{k})$ is called \emph{Gallai-Ramsey-full} if there exists a $k$-edge-colored graph $K_{n}-e$ contains neither a rainbow $K_3$ nor a monochromatic $H_{i}$ for any $i\in[k]$.
So $(H_{1}, H_{2}, \cdots, H_{k})$ is Gallai-Ramsey-full if and only if $gr_{k}^{*}(K_3: H_{1}, \cdots, H_{k})=n-1$.
If $(H_{1}, H_{2}, \cdots, H_{k})$ is Gallai-Ramsey-full and $H=H_{1}=\cdots=H_{k}$, then we say that $H$ is Gallai-Ramsey-full.
In this paper, we investigate the star-critical Gallai-Ramsey numbers for some graphs.
In order to study the star-critical Gallai-Ramsey numbers of a graph $H$, we first characterize the critical graphs on $H$ and then use the critical graphs to find its star-critical Gallai-Ramsey number.
In Section 3, we obtain the star-critical Gallai-Ramsey numbers $gr_{k}^{*}(K_{3}: K_{p_1}, K_{p_2}, \cdots, K_{p_k})$ and $gr_{k}^{*}(K_3: C_4)$.
Thus we find that $(K_{p_1}, K_{p_2}, \cdots, K_{p_k})$ and $C_4$ are Gallai-Ramsey-full.
In Section 4, we get the star-critical Gallai-Ramsey numbers of $P_4$ and $K_{1, m}$.
Thus we find that $P_4$ and $K_{1, m}$ are not Gallai-Ramsey-full.
Finally, for general graphs $H$, we prove the general behavior of $gr_{k}^{*}(K_3: H)$ in Section 5.

\section{Preliminary}

In this section, we list some useful lemmas.
\begin{lemma}{\upshape \cite{Gallai, GyarfasSimonyi, CameronEdmonds}}\label{Lem:G-Part}
For any rainbow triangle free edge-colored complete graph $G$, there exists a partition of $V(G)$ into at least two parts such that there are at most two colors on the edges between the parts and only one color on the edges between each pair of parts. The partition is called a Gallai-partition.
\end{lemma}

\begin{lemma}\label{Lem:qneq3}
Let $G$ be a rainbow triangle free edge-colored complete graph and $(V_1, V_2, \ldots, V_q)$ a Gallai-partition of $V(G)$ with the smallest number of parts.
If $q>2$, then for each part $V_i$, there are exactly two colors on the edges in $E(V_i, V(G)-V_i)$ for $i\in [q]$. Thus $q\neq3$.
\end{lemma}

\begin{proof}
By Lemma~\ref{Lem:G-Part}, suppose, to the contrary, that there exists one part (say $V_{1}$) such that all edges joining $V_{1}$ to other parts are colored by the same color. Then we can find a new Gallai-partition with two parts $(V_{1}, V_{2}\bigcup\cdots \bigcup V_{q})$, which contradicts with that $q$ is smallest.
It follows that $q\neq 3$.
\end{proof}

\begin{lemma}{\upshape \cite{R}}\label{Lem:2P4}
$
R_2(C_{4})=R_2(K_{1, 3})=6, R_2(P_{4})=5.
$
\end{lemma}

\begin{lemma}{\upshape \cite{RRMC}}\label{Lem:C4}
For any positive integer $k$ with that $k\geq 2$,
$
gr_k(K_{3}: C_{4})=k+4.
$
\end{lemma}

\begin{lemma}{\upshape \cite{RRMC}}\label{Lem:P4}
For any positive integer $k$,
$
gr_k(K_{3}: P_{4})=k+3.
$
\end{lemma}

\begin{lemma}{\upshape \cite{GyarfasSimonyi}}\label{Lem:star}
For any $m\geq 3$ and $k\geq 3$,
$$
gr_k(K_{3}: K_{1, m})= \begin{cases}
\frac{5m-6}{2},  & \text{if $m$ is even,}\\
\frac{5m-3}{2},  & \text{if $m$ is odd.}
\end{cases}
$$
\end{lemma}

\section{Gallai-Ramsey-full graphs}

First, we investigate the star-critical Gallai-Ramsey number for complete graphs.

\begin{theorem}\label{Thm:Kp}
For any positive integers $p_1$, $p_2$, $\ldots$, $p_k$ and $k$, $$gr_{k}^{*}(K_{3}: K_{p_1}, K_{p_2}, \cdots, K_{p_k})=gr_{k}(K_{3}: K_{p_1}, K_{p_2}, \cdots, K_{p_k})-1.$$
\end{theorem}

\begin{proof}
Let $n=gr_{k}(K_{3}: K_{p_1}, K_{p_2}, \cdots, K_{p_k})$.
Clearly, $gr_{k}^{*}(K_{3}: K_{p_1}, K_{p_2}, \cdots, K_{p_k})\leq n-1$.
So we only prove that $gr_{k}^{*}(K_{3}: K_{p_1}, K_{p_2}, \cdots, K_{p_k})\geq n-1$.
Since $gr_k(K_3: K_{p_1}, K_{p_2}, \cdots, K_{p_k})=n$, there exists a $k$-edge-colored critical graph $K_{n-1}$ containing neither a rainbow triangle nor a monochromatic $K_{p_i}$ in color $i$ for any $i\in[k]$.
Let $f$ be the $k$-edge-coloring of the critical graph $K_{n-1}$ and $u\in V(K_{n-1})$.
We construct $K_{n-1}\sqcup K_{1, n-2}$ by adding the edge set $\{vw~|~w\in V(K_{n-1})-\{u\}\}$ to the critical graph $K_{n-1}$, where $v$ is the center vertex of $K_{1, n-2}$.
Let $g$ be a $k$-edge-coloring of $K_{n-1}\sqcup K_{1, n-2}$ such that $$
g(e)= \begin{cases}
f(e),  & \text{if $e\in E(K_{n-1})$,}\\
f(uw),  & \text{if $e=vw$ and $w\in V(K_{n-1})-\{u\}$.}
\end{cases}
$$
Clearly, the $k$-edge-colored $K_{n-1}\sqcup K_{1, n-2}$ by $g$ contains neither a rainbow triangle nor a monochromatic $K_{p_i}$ in color $i$ for any $i\in[k]$. Then $gr_{k}^{*}(K_{3}: K_{p_1}, K_{p_2}, \cdots, K_{p_k})\geq n-1$.
Therefore, $gr_{k}^{*}(K_{3}: K_{p_1}, K_{p_2}, \cdots, K_{p_k})=n-1$.
\end{proof}

By Theorem~\ref{Thm:Kp}, we know that $(K_{p_1}, K_{p_2}, \cdots, K_{p_k})$ is Gallai-Ramsey-full.

In the following, we determine the star-critical Gallai-Ramsey number $gr_{k}^{*}(K_{3}: C_{4})$.
By Lemma~\ref{Lem:C4}, $gr_{k}(K_3: C_4)=k+4\geq6$ for any $k\geq 2$.
First we construct a critical graph on $C_4$.

\begin{definition}\label{Def:C4}
Let $k\geq2$, $n=gr_k(K_{3}: C_{4})$, $V(K_{n-1})=\{v_1, v_2, \ldots, v_{n-1}\}$ and color set $[k]$.
The subgraph induced by $\{v_1, v_2, v_3, v_4, v_5\}$ consists of two edge-disjoint $C_5$, say $v_1v_2v_3v_4v_5v_1$ and $v_1v_4v_2v_5v_3v_1$.
Define a $k$-edge-coloring $f$ of $K_{n-1}$ as follows: (1) $f(e)=1$ if $e$ is an edge of $v_1v_2v_3v_4v_5v_1$ and $f(e)=2$ if $e$ is an edge of $v_1v_4v_2v_5v_3v_1$.
(2) For any $j\in \{6, \ldots, n-1\}$ and $i\in [j-1]$, $f(v_jv_i)=j-3$.
We denote this $k$-edge-colored $K_{n-1}$ by $G^{k}_{n-1}$.
\end{definition}

Clearly, graph $G^{k}_{n-1}$ in Definition~\ref{Def:C4} contains neither a rainbow triangle nor a monochromatic $C_4$.

\begin{theorem}\label{Thm:starC4}
For any positive integer $k$ with that $k\geq 2$,
$
gr_{k}^{*}(K_{3}: C_{4})=k+3.
$
\end{theorem}

\begin{proof}
Let $n=gr_k(K_{3}: C_{4})$.
By Lemma~\ref{Lem:C4}, $n=k+4$.
Clearly, $gr_{k}^{*}(K_{3}: C_{4})\leq n-1=k+3$.
So we only prove that $gr_{k}^{*}(K_{3}: C_{4})\geq k+3$.
Let $G^{k}_{n-1}\sqcup K_{1, n-2}$ be a graph obtained from $G^{k}_{n-1}$ by adding the edge set $\{vv_i~|~i\in[n-1]-\{5\}\}$, where $G^{k}_{n-1}$ is described in Definition~\ref{Def:C4} and $v$ is the center vertex of $K_{1, n-2}$.
Let $g$ be a $k$-edge-coloring of $G^{k}_{n-1}\sqcup K_{1, n-2}$ such that
$$
g(e)= \begin{cases}
f(e),  & \text{if $e\in E(G^{k}_{n-1})$, where $f$ is the $k$-edge-coloring in Definition~\ref{Def:C4},}\\
1,  & \text{if $e\in \{vv_2, vv_3\}$,}\\
2,  & \text{if $e\in \{vv_1, vv_4\}$,}\\
i-3,  & \text{if $e=vv_i$ for any $6\leq i\leq n-1$.}
\end{cases}
$$
Clearly, the $k$-edge-colored $G^{k}_{n-1}\sqcup K_{1, n-2}$ by $g$ contains neither a rainbow $K_3$ nor a monochromatic $C_4$. So $gr_{k}^{*}(K_{3}: C_{4})\geq n-1=k+3$.
Therefore, $gr_{k}^{*}(K_{3}: C_{4})=k+3$.
\end{proof}

By Lemma~\ref{Lem:2P4} and Lemma~\ref{Lem:C4}, we know that $C_4$ is Gallai-Ramsey-full and Ramsey-full since $gr_{k}^{*}(K_{3}: C_{4})=gr_{k}(K_{3}: C_{4})-1$ and $r_{*}(C_4)=gr_{2}^{*}(K_{3}: C_{4})=R_2(C_4)-1$.
Note that complete graph $K_n$ is also Ramsey-full and Gallai-Ramsey-full.
So we pose a conjecture as following.\\
{\bf Conjecture.}
Let $H$ be a graph with no isolated vertex.
Then $H$ is Ramsey-full if and only if $H$ is Gallai-Ramsey-full.

\section{Gallai-Ramsey non-full graphs}
In this section, we investigate the star-critical Gallai-Ramsey numbers for $P_4$ and $K_{1, m}$.
Thus, we find that $P_4$ and $K_{1, m}$ are not Gallai-Ramsey-full.
In order to determine the star-critical Gallai-Ramsey number $gr_{k}^{*}(K_{3}: P_{4})$, first we study  the structure of the critical graphs on $P_4$.
By Lemma~\ref{Lem:P4}, $gr_{k}(K_3: P_4)=k+3$ for any positive integer $k$.


\begin{definition}\label{Def:P4}
For any positive integer $k$, let $n=gr_k(K_{3}: P_{4})$, $V(K_{n-1})=\{v_1, v_2, \ldots, v_{n-1}\}$ and color set $[k]$. Define a $k$-edge-coloring $f$ of $K_{n-1}$ as follows: (1) $f(v_1v_2)=f(v_1v_3)=f(v_2v_3)=1$; (2) $f(v_1v_i)=f(v_2v_i)=f(v_3v_i)=i-2$ for any $4\leq i\leq n-1$; (3) For any $4\leq i<j\leq n-1$, $f(v_iv_j)=i-2$ or $j-2$ such that there is no rainbow triangle in the subgraph induced by $\{v_4, \ldots, v_{n-1}\}$.
The $k$-edge-coloring $f$ is called $k$-critical coloring on $P_4$ and let $\mathcal{G}^{k}_{n-1}=\{$all $k$-critical colored $K_{n-1}$ on $P_4\}$.
\end{definition}

Clearly, $\mathcal{G}^{k}_{n-1}\neq\emptyset$. For example, we set $f(v_iv_j)=j-2$ for any $4\leq i< j\leq n-1$. It is easy to check that $f$ is a $k$-critical coloring on $P_4$.
On the other hand, the graphs in Definition~\ref{Def:P4} are not unique when $k\geq3$.
For example, if $k=4$, we can set $f(v_4v_5)=f(v_4v_6)=2$ and $f(v_5v_6)=3$ or $f(v_4v_6)=f(v_5v_6)=4$ and $f(v_4v_5)=2$.
It is easy to check that the two colorings are $k$-critical.


\begin{proposition}\label{Pro:P4}
For any positive integer $k$, let $n=gr_k(K_{3}: P_{4})$. Then $H$ is a $k$-edge-colored complete graph $K_{n-1}$ containing neither a rainbow triangle nor a monochromatic $P_4$ if and only if $H\in\mathcal{G}^{k}_{n-1}$, where $\mathcal{G}^{k}_{n-1}$ is described in Definition~\ref{Def:P4}.
\end{proposition}

\begin{proof}
It suffices to prove the 'necessity'.
Let $H$ be a $k$-edge-colored $K_{n-1}$ containing neither a rainbow triangle nor a monochromatic $P_4$.
Then, by Lemma~\ref{Lem:G-Part}, there exists a Gallai-partition of $V(H)$. Choose a Gallai-partition with the smallest number of parts, say $(V_{1}, V_{2}, \cdots, V_{q})$.
Then $q\ge 2$ and $q\neq 3$ by Lemma~\ref{Lem:G-Part} and Lemma~\ref{Lem:qneq3}.
Let $H_{i}=H[V_{i}]$ for each part $V_i$.
Since $R_2(P_4)=5$ by Lemma~\ref{Lem:2P4}, we have that $2\leq q\leq 4$.
If there exist two parts $V_i$, $V_j$ such that $|V_i|\geq 2$ and $|V_j|\geq 2$, then there is a monochromatic $P_4$, a contradiction.
So there is at most one part with at least two vertices and all other parts are single vertex.
W.L.O.G., suppose that $|H_2|=\ldots=|H_q|=1$.

\begin{claim}\label{Cla:q=2}
If $k\geq3$, then $q=2$.
\end{claim}

\noindent{\bf Proof.}
Suppose, to the contrary, that $q=4$.
Then $|H_1|\geq 2$ since $k\geq 3$.
By the pigeonhole principle, there are two single vertex parts such that all edges between the two parts and $V_1$ are in the same color.
So there is a monochromatic $P_4$, a contradiction.
Then $q=2$.

\begin{claim}\label{Cla:mk3}
$H$ contains a monochromatic $K_3$.
\end{claim}

\noindent{\bf Proof.}
We prove this claim by induction on $k$.
When $k=1$, it is trivial.
When $k=2$, $H=K_4$.
It is easy to check that $H$ has a monochromatic $K_3$ since $H$ contains no monochromatic $P_4$.
Suppose that $k\geq 3$ and the claim holds for any $k'$ such that $k'<k$.
By Claim~\ref{Cla:q=2}, $q=2$.
Hence, $|H_1|\geq 4$.
W.L.O.G., suppose that the edges between the two parts are colored by 1.
To avoid a monochromatic $P_4$ in color 1, $H_1$ contains no edges colored by 1.
By the induction hypothesis, $H_1$ contains a monochromatic $K_3$. So $H$ contains a monochromatic $K_3$.

By Claim~\ref{Cla:mk3}, W.L.O.G., we can assume that the monochromatic $K_3$ is $v_1v_2v_3$ and this $K_3$ is colored by 1.
Let $V(H)-\{v_1, v_2, v_3\}=\{v_4, \ldots, v_{n-1}\}$.
Then, to avoid a monochromatic $P_4$, the colors of edges between $\{v_4, \ldots, v_{n-1}\}$ and $\{v_1, v_2, v_3\}$ are not 1.
So to avoid a rainbow triangle, all edges in $E(v_i, \{v_1, v_2, v_3\})$ are in the same color for any $i\geq4$.
Then, to avoid a monochromatic $P_4$, the edges in $E(v_i, \{v_1, v_2, v_3\})$ and the edges in $E(v_j, \{v_1, v_2, v_3\})$ have different colors for any $4\leq i<j\leq n-1$.
W.L.O.G., we can assume that the edges in $E(v_i, \{v_1, v_2, v_3\})$ have color $i-2$ for any $i\geq4$.
Then to avoid a rainbow triangle, the edge $v_iv_j$ is colored by $i-2$ or $j-2$ for any $4\leq i<j\leq n-1$.
Thus, $H\in\mathcal{G}^{k}_{n-1}$, where $\mathcal{G}^{k}_{n-1}$ is described in Definition~\ref{Def:P4}.
\end{proof}

\begin{theorem}\label{Thm:starP4}
For any positive integer $k$,
$
gr_{k}^{*}(K_{3}: P_{4})=k.
$
\end{theorem}

\begin{proof}
Let $H\in \mathcal{G}^{k}_{n-1}$, where $\mathcal{G}^{k}_{n-1}$ is defined in Definition~\ref{Def:P4} and $n=gr_k(K_{3}: P_{4})=k+3$.
First we show that $gr_{k}^{*}(K_{3}: P_{4})\geq k$.
When $k=1$, $H$ is a monochromatic $K_3$. Then $gr_{1}^{*}(K_{3}: P_{4})\geq 1$.
So we can assume that $k\geq 2$.
By Definition~\ref{Def:P4}, $H$ contains a monochromatic $K_{3}=v_1v_2v_3$.
Then we construct $H\sqcup K_{1, k-1}$ by adding the edge set $\{vv_i~|~4\leq i\leq n-1\}$ to $H$, where $v$ is the center vertex of $K_{1, k-1}$.
Let $c$ be a $k$-edge-coloring of $H\sqcup K_{1, k-1}$ such that $$
c(e)= \begin{cases}
f(e),  & \text{if $e\in E(H)$,}\\
f(v_1v_i),  & \text{if $e=vv_i$ for any $4\leq i\leq n-1$,}
\end{cases}
$$ where $f$ is the $k$-edge-coloring in Definition~\ref{Def:P4}.
Clearly, this $k$-edge-colored $H\sqcup K_{1, k-1}$ by $c$ contains neither a rainbow triangle nor a monochromatic $P_4$ and $d(v)=k-1$.
Hence, $gr_{k}^{*}(K_{3}: P_{4})\geq k$.

Now we show that $gr_{k}^{*}(K_{3}: P_{4})\leq k$.
Let $G=H\sqcup K_{1, k}$ be a $k$-edge-colored graph and $v$ be the center vertex of $K_{1, k}$.
Since $d_{G}(v)=k$ and $|H|=k+2$ by Definition~\ref{Def:P4}, there is at least one edge of $G$ between $\{v_1, v_2, v_3\}$ and $v$.
W.L.O.G., suppose that $v_1v\in E(G)$ and $v_1v$ is colored by $i$, where $i\in [k]$.
Then there is a monochromatic $P_4=vv_1v_{i+2}v_2$ in color $i$.
Hence, $gr_{k}^{*}(K_{3}: P_{4})\leq k$.
Therefore, we have that $gr_{k}^{*}(K_{3}: P_{4})=k$.
\end{proof}

In the following, we determine the star-critical Gallai-Ramsey number $gr_{k}^{*}(K_{3}: K_{1, m})$. First we characterize all critical graphs on star $K_{1, m}$.
By Lemma~\ref{Lem:star}, for any $k\geq3$ and any $m\geq3$, $$
gr_k(K_{3}: K_{1, m})= \begin{cases}
\frac{5m-6}{2},  & \text{if $m$ is even,}\\
\frac{5m-3}{2},  & \text{if $m$ is odd.}
\end{cases}
$$


\begin{definition}\label{Def:star}
Let $m\geq 3$, $k\geq 3$, $n=gr_k(K_{3}: K_{1, m})$ and color set $[k]$.
For the complete graph $K_{n-1}$, choose a partition $(V_{1}$, $V_{2}$, $V_3$, $V_4$, $V_{5})$ of $V(K_{n-1})$ such that $|V_i|=\frac{m-1}{2}$ for $i\in [5]$ if $m$ is odd and $|V_1|=\frac{m}{2}$, $|V_i|=\frac{m-2}{2}$ for $i\geq 2$ if $m$ is even.
Let $H_i$ be the induced subgraph of $K_{n-1}$ by $V_i$ and $v_i\in V_i$ for any $i\in [5]$.
The subgraph induced by $\{v_1, v_2, v_3, v_4, v_5\}$ consists of two edge-disjoint $C_5$, say $v_1v_2v_3v_4v_5v_1$ and $v_1v_4v_2v_5v_3v_1$.
Define a $k$-edge-coloring $f$ of $K_{n-1}$ as follows:
(1) $f(e)=1$ if $e$ is an edge of $v_1v_2v_3v_4v_5v_1$ and $f(e)=2$ if $e$ is an edge of $v_1v_4v_2v_5v_3v_1$.
Color all edges between $V_i$ and $V_j$ by the color of $v_iv_j$ for any $1\leq i< j\leq 5$ (see Fig. 1(c)).
(2) Color $H_i$ as follows such that each $H_i$ contains no rainbow triangle.
If $m$ is odd, then we color each $H_i$ by the color set $\{3, \ldots, k\}$.
Let $m$ be even. Then we color $H_1$ by the color set $[k]$ such that the subgraph induced by $E_1=\{e\in E(H_1)~|~f(e)=1, 2\}$ is a matching or $E_1=\emptyset$;
We color $H_2$ and $H_5$ by color set $\{2, \ldots, k\}$ such that the subgraph induced by $E_2=\{e\in E(H_2)\cup E(H_5)~|~f(e)=2\}$ is a matching or $E_2=\emptyset$;
We color $H_3$ and $H_4$ by color set $\{1, 3, \ldots, k\}$ such that the subgraph induced by $E_3=\{e\in E(H_3)\cup E(H_4)~|~f(e)=1\}$ is a matching or $E_3=\emptyset$.
The $k$-edge-coloring $f$ is called $k$-critical coloring on star $K_{1, m}$ and let $\mathcal{S}^{k}_{n-1}=\{$all $k$-critical colored $K_{n-1}$ on star $\}$.
\end{definition}

Clearly, $\mathcal{S}^{k}_{n-1}\neq\emptyset$ and each graph in $\mathcal{S}^{k}_{n-1}$ in Definition~\ref{Def:star} contains neither a rainbow triangle nor a monochromatic $K_{1, m}$.

\begin{proposition}\label{Pro:star}
Let $k\geq 3$, $n=gr_k(K_{3}: K_{1, m})$, $m\geq3$ if $m$ is odd and $m\geq12$ otherwise.
Then $H$ is a $k$-edge-colored $K_{n-1}$ containing neither a rainbow triangle nor a monochromatic $K_{1, m}$ if and only if $H\in\mathcal{S}^{k}_{n-1}$, where $\mathcal{S}^{k}_{n-1}$ is described in Definition~\ref{Def:star}.
\end{proposition}

\begin{proof}
It suffices to prove the 'necessity'.
Let $H$ be a $k$-edge-colored $K_{n-1}$ containing neither a rainbow triangle nor a monochromatic $K_{1, m}$. Then, by Lemma~\ref{Lem:G-Part}, there exists a Gallai-partition of $V(H)$.
Choose a Gallai-partition with the smallest number of parts, say $(V_{1}, V_{2}, \cdots, V_{q})$, where $|V_1|\geq |V_2|\geq\ldots\geq |V_q|$. Let $H_{i}=H[V_{i}]$ and $v_i\in V_i$ for each $i\in [q]$.
Then $q\ge 2$ and $q\neq 3$ by Lemma~\ref{Lem:G-Part} and Lemma~\ref{Lem:qneq3}.
W.L.O.G., let 1 and 2 be the colors of edges between parts of the partition.
Suppose that $|V_q|\leq\frac{m-3}{2}$ if $m$ is odd and $|V_q|\leq\frac{m-6}{2}$ if $m$ is even.
Then $|V(H)-V_q|\geq 2m-1$.
Let $v\in V_q$. By the pigeonhole principle, there are at least $m$ edges in $E(v, V(H)-V_q)$ with the same color.
It implies a monochromatic $K_{1, m}$, a contradiction.
Then $|V_i|\geq \frac{m-1}{2}$ for any $i\in [q]$ if $m$ is odd and $|V_i|\geq \frac{m-4}{2}$ for any $i\in [q]$ if $m$ is even.

\begin{claim}\label{Cla:K13}
The subgraph induced by $\{v_1, \ldots, v_q\}$ contains no monochromatic $K_{1, 3}$.
\end{claim}

\noindent{\bf Proof.} Suppose, to the contrary, that the subgraph induced by $\{v_1, \ldots, v_q\}$ contains a monochromatic $K_{1, 3}$.
It follows that there exist four parts, say $V_i$, $V_{j_1}$, $V_{j_2}$ and $V_{j_3}$, such that all edges in $E(V_i, V_{j_1}\cup V_{j_2}\cup V_{j_3})$ are the same color.
Since $3\times\frac{m-1}{2}\geq m$ if $m\geq3$ is odd and $3\times\frac{m-4}{2}\geq m$ if $m\geq12$ is even, it follows that there is a monochromatic $K_{1, m}$, a contradiction.

\begin{claim}\label{Cla:q=5}
q=5.
\end{claim}

\begin{figure}
	\begin{center}
    \includegraphics[scale=0.5]{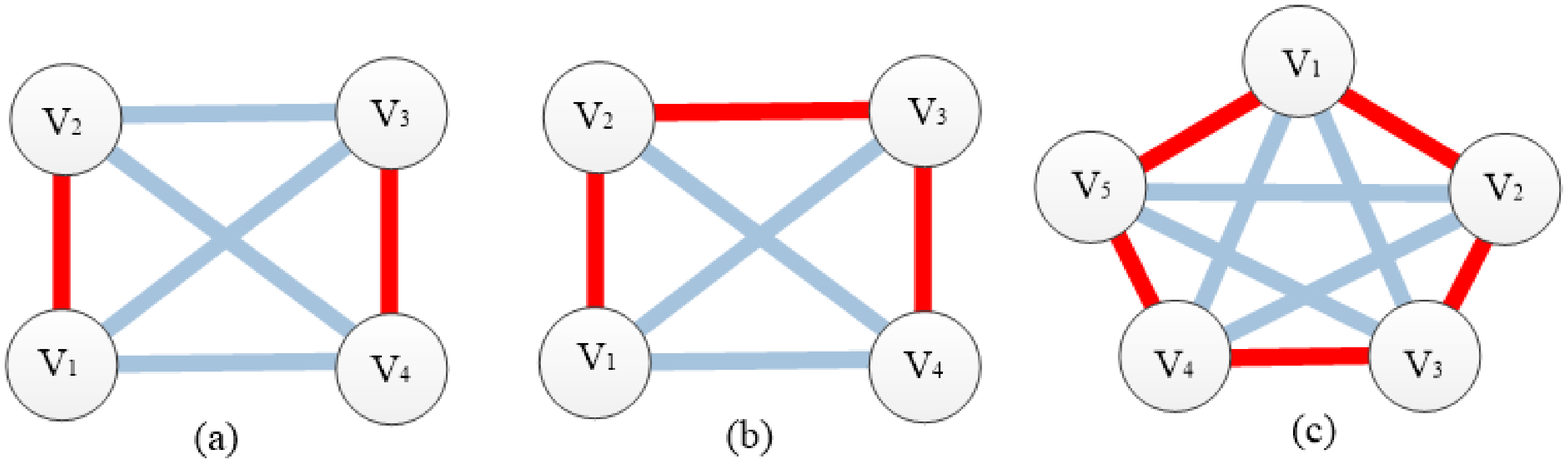}\\
		\caption{}
		\label{Fig:1}
	\end{center}
\end{figure}

\noindent{\bf Proof.}
If $q=2$, then $|V_1|\geq\frac{|H|}{2}\geq m$, which implies that $H$ has a monochromatic $K_{1, m}$, a contradiction.
Now suppose that $q=4$.
By Lemma~\ref{Lem:qneq3}, for each part $V_i$, there are exactly two colors on the edges in $E(V_i, V(G)-V_i)$ for $i\in [4]$.
Thus, the subgraph induced by $\{v_1, v_2, v_3, v_4\}$ must be the graph as shown in Fig.1 (a) or Fig.1 (b).
By the minimality of $q$, the subgraph induced by $\{v_1, v_2, v_3, v_4\}$ consists of two edge-disjoint $P_4$ such that one is colored by 1 and the other is colored by 2.
W.L.O.G., let $v_1v_2v_3v_4$ be colored by 1 and $v_3v_1v_4v_2$ colored by 2 (see Fig.1 (b)).
Since $|V_1|+|V_2|\geq\frac{|H|}{2}\geq m$, it follows that the induced subgraph by $E(V_1\cup V_2, V_4)$ contains a monochromatic star $K_{1, m}$ in color 2, a contradiction.
So $q\geq 5$.
By Claim~\ref{Cla:K13}, the subgraph induced by $\{v_1, \ldots, v_q\}$ contains no monochromatic $K_{1, 3}$.
Since $R_2(K_{1, 3})=6$ by Lemma~\ref{Lem:2P4}, we have that $q\leq 5$.
So $q=5$.

By Claim~\ref{Cla:K13}, the subgraph induced by $\{v_1, \ldots, v_5\}$ contains no monochromatic $K_{1, 3}$.
So this subgraph consists of two edge-disjoint $C_5$ such that one is colored by 1 and the other is colored by 2.
W.L.O.G., let $v_1v_2v_3v_4v_5v_1$ be colored by 1 and $v_1v_4v_2v_5v_3v_1$ colored by 2 (see Fig.1 (c)).
When $m\geq3$ is odd, since $\frac{|H|}{5}=\frac{m-1}{2}$ and every part has order at least $\frac{m-1}{2}$, it follows that each part contains exactly $\frac{m-1}{2}$ vertices.
To avoid a monochromatic $K_{1, m}$, $H_i$ contains no edge with color 1 or 2 for any $i\in[5]$.
Hence, for any $i\in [5]$, $H_i$ is colored by color set $\{3, \ldots, k\}$.
When $m\geq 12$ is even, if there is a part with $\frac{m-4}{2}$ vertices, then there exists a monochromatic $K_{1, m}$ since there must be two parts, say $V_1$ and $V_2$, such that $|V_1|+|V_2|\geq m$.
Hence, every part has at least $\frac{m-2}{2}$ vertices.
Since $|H|=\frac{5m-8}{2}$, we have that $|V_1|=\frac{m}{2}$ and $|V_2|=\ldots=|V_5|=\frac{m-2}{2}$.
To avoid a monochromatic $K_{1, m}$, if $H_1$ contains edges colored by 1 or 2, then $\{e\in E(H_1)~|~ f(e)=1, 2\}$ is a matching.
To avoid a monochromatic $K_{1, m}$, $H_2$ and $H_5$ contain no edges colored by 1. If there exist edges colored by 2 in $H_2$ or $H_5$, then $\{e\in E(H_2)\cup E(H_5)~|~f(e)=2\}$ is a matching.
Similarly, $H_3$ and $H_4$ contain no edges colored by 2. If there exist edges colored by 1 in $H_3$ or $H_4$, then $\{e\in E(H_3)\cup E(H_4)~|~f(e)=1\}$ is a matching.
Thus, the graph $H\in\mathcal{S}^{k}_{n-1}$, where $\mathcal{S}^{k}_{n-1}$ is described in Definition~\ref{Def:star}.
\end{proof}

\begin{theorem}\label{Thm:star}
For any $k\geq 3$,
$
gr_k^{*}(K_{3}: K_{1, m})= \begin{cases}
2m-2,  & \text{if $m\geq 12$ is even,}\\
m,  & \text{if $m\geq 3$ is odd.}
\end{cases}
$
\end{theorem}

\begin{proof}
Let $n=gr_k(K_{3}: K_{1, m})$.
Now we consider the following cases.

\begin{case}\label{case:modd}
$m\geq3$ is odd.
\end{case}

Let $H\in \mathcal{S}^{k}_{n-1}$, where $\mathcal{S}^{k}_{n-1}$ is defined in Definition~\ref{Def:star}.
First we show that $gr_{k}^{*}(K_{3}: K_{1, m})\geq m$.
We construct $H\sqcup K_{1, m-1}$ by adding the edge set $\{vv_i~|~v_i\in V_1\cup V_2\}$ to $H$, where $v$ is the center vertex of $K_{1, m-1}$.
Let $c$ be a $k$-edge-coloring of $H\sqcup K_{1, m-1}$ such that $$
c(e)= \begin{cases}
f(e),  & \text{if $e\in E(H)$, where $f$ is the $k$-edge-coloring in Definition~\ref{Def:star},}\\
3,  & \text{if $e\in E(v, V(H))$.}
\end{cases}
$$ Clearly, this $k$-edge-colored graph $H\sqcup K_{1, m-1}$ by $c$ contains neither a rainbow triangle nor a monochromatic $K_{1, m}$ and $d(v)=m-1$.
Hence, $gr_{k}^{*}(K_{3}: K_{1, m})\geq m$.
Now we show that $gr_{k}^{*}(K_{3}: K_{1, m})\leq m$.
Let $G=H\sqcup K_{1, m}$ be a $k$-edge-colored graph and $v$ the center vertex of $K_{1, m}$.
Suppose that $G$ contains no rainbow triangle.
Then we prove that $G$ contains a monochromatic $K_{1, m}$.
By the definition of $H$, if there exists an edge $e\in E(v, V(H))$ in color 1 or 2, then there is a monochromatic $K_{1, m}$ in color 1 or 2.
So we can assume that all edges in $E(v, V(H))$ have color in $\{3, \ldots, k\}$.
Since $H$ consists of 5 parts such that every part has exactly $\frac{m-1}{2}$ vertices and $|E(v, V(H))|=m$, we have that there are at least 3 parts of $H$, say $V_1$, $V_2$, $V_3$, such that $E(v, V_i)\neq\emptyset$ for any $i\in [3]$.
Then to avoid a rainbow triangle, all edges in $E(v, V(H))$ have the same color.
So there is a monochromatic $K_{1, m}$ with the center vertex $v$.
Hence, $gr_{k}^{*}(K_{3}: K_{1, m})\leq m$.
Therefore, we have that $gr_{k}^{*}(K_{3}: K_{1, m})=m$.

\begin{case}\label{case:meven}
$m\geq12$ is even.
\end{case}

First we show that $gr_{k}^{*}(K_{3}: K_{1, m})\geq 2m-2$.
Let $H\in \mathcal{S}^{k}_{n-1}$ such that each subgraph $H_i$ of $H$ is colored by the color set $\{3, \ldots, k\}$, where $\mathcal{S}^{k}_{n-1}$ is defined in Definition~\ref{Def:star}.
We construct $H\sqcup K_{1, 2m-3}$ by adding the edge set $\{vv_i~|~v_i\in V_1\cup V_2\cup V_3\cup V_5\}$ to $H$, where $v$ is the center vertex of $K_{1, 2m-3}$.
Let $c$ be a $k$-edge-coloring of $H\sqcup K_{1, 2m-3}$ such that $$
c(e)= \begin{cases}
f(e),  & \text{if $e\in E(H)$, where $f$ is the $k$-edge-coloring in Definition~\ref{Def:star},}\\
1,  & \text{if $e\in E(v, V_1\cup V_3)$,}\\
2,  & \text{if $e\in E(v, V_2\cup V_5)$.}\\
\end{cases}
$$ Clearly, this $k$-edge-colored graph $H\sqcup K_{1, 2m-3}$ by $c$ contains neither a rainbow triangle nor a monochromatic $K_{1, m}$ and $d(v)=2m-3$.
Hence, $gr_{k}^{*}(K_{3}: K_{1, m})\geq 2m-2$.

Now we show that $gr_{k}^{*}(K_{3}: K_{1, m})\leq 2m-2$.
Let $H\in \mathcal{S}^{k}_{n-1}$, $G=H\sqcup K_{1, 2m-2}$ be a $k$-edge-colored graph and $v$ be the center vertex of $K_{1, 2m-2}$.
Suppose that $G$ contains no rainbow triangle.
We prove that $G$ has a monochromatic $K_{1, m}$ in the following.

First, suppose that there exists an edge $e\in E(v, V(H))$ in color $\{3, \ldots, k\}$, say $e$ is colored by 3.
W.L.O.G., let $e\in E(v, V_1)$ or $e\in E(v, V_4)$ by the symmetry of $V_2$, $V_3$, $V_4$ and $V_5$ in $H$ (see Fig.1(c)).
First we assume that $e\in E(v, V_1)$. Then to avoid a rainbow triangle, all edges in $E(v, V_2\cup V_5)$ are colored by 1 or 3.
If there exists an edge in $E(v, V_2\cup V_5)$ colored by 1, then there is a monochromatic $K_{1, m}$ in color 1.
So we can assume that all edges in $E(v, V_2\cup V_5)$ are colored by 3.
Then to avoid a rainbow triangle, all edges in $E(v, V_3\cup V_4)$ are colored by 3.
So there is a monochromatic $K_{1, m}$ in color 3.
Now we assume that $e\in E(v, V_4)$.
Then to avoid a rainbow triangle, all edges in $E(v, V_5)$ are colored by 1 or 3.
If there exists an edge in $E(v, V_5)$ colored by 1, then there is a monochromatic $K_{1, m}$ in color 1.
So we can assume that all edges in $E(v, V_5)$ are colored by 3.
Then to avoid a rainbow triangle, all edges $E(v, V_3)$ must be colored by 3.
It follows that all edges in $E(v, V_1\cup V_2)$ are colored by 3 since $G$ contains no rainbow triangle.
Hence, there is a monochromatic $K_{1, m}$ in color 3.

Now we can assume that all edges in $E(v, V(H))$ are colored by 1 or 2.
By the definition of $H$ in Definition~\ref{Def:star}, for any part $V_i$, $E(v, V_i)\neq\emptyset$ since $|E(v, V(H))|=2m-2$.
If there exists an edge in $E(v, V_2\cup V_5)$ colored by 1, then there is a monochromatic $K_{1, m}$ in color 1.
Similarly, if there exists an edge in $E(v, V_3\cup V_4)$ colored by 2, then there is a monochromatic $K_{1, m}$ in color 2.
Thus, suppose that all edges in $E(v, V_2\cup V_5)$ are colored by 2 and all edges in $E(v, V_3\cup V_4)$ are colored by 1.
Let $a, b\in V_1$. Suppose that the color of $va$ is 1 and the color of $vb$ is 2.
To avoid a rainbow triangle, the edge $ab$ must be colored by 1 or 2.
Then there is a monochromatic $K_{1, m}$ in color 1 or 2.
Hence, suppose that all edges in $E(v, V_1)$ are in the same color which is either 1 or 2.
Thus, there are $m$ edges in $E(v, V_1\cup V_3\cup V_4)$ colored by 1 or there are $m$ edges in $E(v, V_1\cup V_2\cup V_5)$ colored by 2, which follows a monochromatic $K_{1, m}$.
Hence, $gr_{k}^{*}(K_{3}: K_{1, m})\leq 2m-2$.
Therefore, $gr_{k}^{*}(K_{3}: K_{1, m})=2m-2$.
\end{proof}

\noindent{\bf Remark.}
For any $m\geq3$ is odd and $m\geq12$ is even, $gr_{k}^{*}(K_3: K_{1, m})$ is determined by Theorem~\ref{Thm:star}.
When $m=4$, $gr_{k}(K_3: K_{1, 4})=7$ by Lemma~\ref{Lem:star}.
We can verify that $gr_{k}^{*}(K_3: K_{1, 4})=6$ (the proof is omitted).
When $m=6, 8, 10$, the problem to determine $gr_{k}^{*}(K_3: K_{1, m})$ is open.

\section{Star-critical Gallai-Ramsey numbers for general graphs}

In 2010, Gy\'{a}rf\'{a}s et al. \cite{AGAS} proved the general behavior of $gr_{k}(K_3: H)$.
It turns out that for some graphs $H$, the order of magnitude of $gr_{k}(K_3: H)$ seems hopelessly difficult to determine.
So finding the exact value of $gr_{k}^{*}(K_3: H)$ is far from trivial.
Thus, in this section we investigate the general behavior of star-critical Gallai-Ramsey numbers for general graphs.
For general bipartite graph $H$, we can only consider the case that $H$ is not a star since the exact value of $gr_{k}^{*}(K_3: K_{1, m})$ is determined in Theorem~\ref{Thm:star}.
First give the following lemmas and definitions.

For a connected bipartite graph $H$, define $s(H)$ to be the order of the smaller part of $H$ and $l(H)$ to be the order of the larger part.
For a connected non-bipartite graph $H$, call a graph $H^{'}$ a \emph{merge} of $H$ if $H^{'}$ can be obtained from $H$ by identifying some independent sets of $H$ (and removing any resulting repeated edges).
Let $\mathscr{H}$ be the set of all possible merges of $H$ and $R_{2}(\mathscr{H})$ the minimum integer $n$ such that every 2-edge-colored $K_n$ contains a monochromatic graph in $\mathscr{H}$.
Then there exists a 2-edge-colored critical graph $K_{n-1}$ containing no monochromatic graph in $\mathscr{H}$.
Let $m(H)=R_{2}(\mathscr{H})$ and $r_{*}(\mathscr{H})$ be the smallest integer $r$ such that every 2-edge-colored $K_{m(H)-1}\sqcup K_{1, r}$ contains a monochromatic graph in $\mathscr{H}$.
Then there exists a 2-edge-colored critical graph $K_{m(H)-1}\sqcup K_{1, r-1}$ containing no monochromatic graph in $\mathscr{H}$.
The \emph{chromatic number} of a graph $G$, denoted by $\chi(G)$, is the smallest number of colors needed to color the vertices of $G$ so that no two adjacent vertices share the same color.


\begin{lemma}{\upshape \cite{WMSX}}\label{Lem:bipartite}
Let $H$ be a connected bipartite graph and $k$ an integer such that $k\geq2$. Then
$$gr_{k}(K_3: H)\geq R_{2}(H)+(k-2)(s(H)-1).$$
\end{lemma}

To prove the above lower bound, Wu et al. in \cite{WMSX} gave the following definition.

\begin{definition}\cite{WMSX}\label{Def:bipartite}
Let $k\geq2$ and $n=R_{2}(H)+(k-2)(s(H)-1)$.
For the complete graph $K_{n-1}$, choose a partition $(V_1, \ldots, V_{k-1})$ of $V(K_{n-1})$ such that $|V_1|=R_2(H)-1$ and $|V_i|=s(H)-1$ for $2\leq i\leq k-1$ and let $H_i$ be the subgraph of $K_{n-1}$ induced by $V_i$ for $i\in[k-1]$.
Define a $k$-edge-coloring $f$ of $K_{n-1}$ as follows:
(1) Color $H_1$ by colors 1 and 2 such that $H_1$ is a 2-edge-colored critical graph containing no monochromatic $H$.
(2) Color $H_i$ by $i+1$ for any $i\in \{2, \ldots, k-1\}$.
(3) For any $j\in \{2, \ldots, k-1\}$ and $i\in[j-1]$, $f(v_iv_j)=j+1$, where $v_i\in V_i$ and $v_j\in V_j$.
We denote this $k$-edge-colored $K_{n-1}$ by $B_{n-1}^{k}$.
\end{definition}

Clearly, $B_{n-1}^{k}$ in Definition~\ref{Def:bipartite} contains no rainbow triangle and no monochromatic $H$.

\begin{lemma}{\upshape \cite{M}}\label{Lem:nonbipartite}
Let $H$ be a connected non-bipartite graph and $k$ an integer such that $k\geq2$. Then
$$gr_{k}(K_3: H)\geq \begin{cases}
(R_2(H)-1)\cdot (m(H)-1)^{(k-2)/2}+1,  & \text{if $k$ is even,}\\
(\chi(H)-1)\cdot (R_2(H)-1)\cdot (m(H)-1)^{(k - 3)/2}+1,  & \text{if $k$ is odd.}
\end{cases}$$
\end{lemma}

To prove the above lower bound, Magnant in \cite{M} gave the following definition. Recall a \emph{blow-up} of an edge-colored graph $G$ on a graph $H$ is a new graph obtained from $G$ by replacing each vertex of $G$ with $H$ and replacing each edge $e$ of $G$ with a monochromatic complete bipartite graph $(V(H), V(H))$ in the same color with $e$.

\begin{definition}\cite{M}\label{Def:nonbipartite}
Let $k\geq2$ and $$n_k=\begin{cases}
(R_2(H)-1)\cdot (m(H)-1)^{(k-2)/2}+1,  & \text{if $k$ is even,}\\
(\chi(H)-1)\cdot (R_2(H)-1)\cdot (m(H)-1)^{(k - 3)/2}+1,  & \text{if $k$ is odd.}
\end{cases}$$
Define a $k$-edge-colored $K_{n_k-1}$ by induction on $k$, which is denoted by $N_{n_{k}-1}^{k}$.
(1)  $N_{n_{2}-1}^{2}$ is a 2-edge-colored critical graph (using colors 1 and 2) with $R_{2}(H)-1$ vertices containing no monochromatic $H$.
Suppose that for any $2i<k$, $N_{n_{2i}-1}^{2i}$ has been constructed which contains no rainbow triangle and no monochromatic $H$.
(2) If $2i+2\leq k$, then let $D$ be a 2-edge-colored critical graph $K_{m(H)-1}$ (using colors $2i+1$ and $2i+2$) containing no monochromatic $H'$ in $\mathscr{H}$.
Then construct $N_{n_{2i+2}-1}^{2i+2}$ by making a blow-up of $D$ on $N_{n_{2i}-1}^{2i}$.
(3) If $2i+1=k$, then construct $N_{n_{2i+1}-1}^{2i+1}$ by making a blow-up of $K_{\chi(H)-1}$ on $N_{n_{2i}-1}^{2i}$, where $K_{\chi(H)-1}$ is colored by $2i+1$.
\end{definition}

Clearly, $N_{n_{k}-1}^{k}$ in Definition~\ref{Def:nonbipartite} contains no rainbow triangle and no monochromatic $H$.

\begin{lemma}{\upshape \cite{AGAS}}\label{Lem:graph}
Let $k$ be a positive integer. Then for any connected bipartite graph $H$,
$gr_{k}(K_3: H)\leq (R_{2}(H)-1)\cdot[(l(H)-1)k+2]\cdot(l(H)-1),$ and for any connected
non-bipartite graph $H$, $gr_{k}(K_3: H)\leq (R_{2}(H)-1)^{k\cdot(|H|-1)+1}.$
\end{lemma}

\begin{theorem}\label{Thm:graph}
Let $H$ be a graph with no isolated vertex. If $H$ is bipartite and not a star, then $gr_k^{*}(K_3: H)$ is linear in $k$.
If $H$ is not bipartite, then $gr_k^{*}(K_3: H)$ is exponential in $k$.
\end{theorem}

\begin{proof}
First, suppose that $H$ is a bipartite graph but not a star.
Now we prove that $gr_k^{*}(K_3: H)\geq r_{*}(H)+(k-2)(s(H)-1)$ which is linear in $k$.
Let $t=r_{*}(H)-1+(k-2)(s(H)-1)$ and $B_{n-1}^{k}$ be the graph defined in Definition~\ref{Def:bipartite}, where $n=R_{2}(H)+(k-2)(s(H)-1)$.
Then $H_1$ is the subgraph of $B_{n-1}^{k}$ such that $|H_1|=R_2(H)-1$ and $H_1$ is a 2-edge-colored critical graph containing no monochromatic $H$.
Hence, we can find a 2-edge-colored critical graph $H_{1}\sqcup K_{1, r_{*}(H)-1}$ (using colors 1 and 2) containing no monochromatic $H$.
Denote this 2-edge-coloring by $h$ and the center vertex of $K_{1, r_{*}(H)-1}$ by $v$.
Let $\{x_1, \ldots, x_{r_{*}(H)-1}\}\subseteq V_1$ and $vx_i\in E(H_{1}\sqcup K_{1, r_{*}(H)-1})$ for any $1\leq i\leq r_{*}(H)-1$.
We construct $B_{n-1}^{k}\sqcup K_{1, t}$ by adding the edge set $E^{*}=\bigcup\limits_{i=2}^{k-1}\{vu~|~u\in V_i\}\cup \{vx_1, \ldots, vx_{r_{*}(H)-1}\}$ to $B_{n-1}^{k}$, where $V_2$, $\ldots$, $V_{k-1}$ are the parts of $V(B_{n-1}^{k})$ in Definition~\ref{Def:bipartite}.
Let $g$ be a $k$-edge-coloring of $B_{n-1}^{k}\sqcup K_{1, t}$ such that
$$g(e)= \begin{cases}
f(e),  & \text{if $e\in E(B_{n-1}^{k})$, where $f$ is defined in Definition~\ref{Def:bipartite}, }\\
h(e),  & \text{if $e\in E(v, V_1)$,}\\
i+1,  & \text{if $e\in E(v, V_i)$ for any $i\in \{2, \ldots, k-1\}$.}\\
\end{cases}
$$
Thus, $d(v)=r_{*}(H)-1+(k-2)(s(H)-1)$ and the graph $B_{n-1}^{k}\sqcup K_{1, t}$ colored by $g$ contains neither a rainbow triangle nor a monochromatic $H$.
Hence, $gr_{k}^{*}(K_{3}: H)\geq r_{*}(H)+(k-2)(s(H)-1)$.
For the upper bound, by Lemma~\ref{Lem:graph}, we know that $gr_{k}(K_3: H)\leq (R_2(H)-1)\cdot [(l(H)-1)k+2]\cdot(l(H)-1)$. Clearly, $gr_{k}^{*}(K_{3}: H)\leq (R_2(H)-1)\cdot [(l(H)-1)k+2]\cdot(l(H)-1)-1$.
Thus, for bipartite graph $H$, we know that there exists a lower bound and an upper bound of $gr^{*}_{k}(K_3: H)$ linear in $k$.
It follows that the statement in Theorem~\ref{Thm:graph} holds when $H$ is bipartite.

Now, suppose that $H$ is a connected non-bipartite graph.
Then we give a lower bound on $gr_k^{*}(K_3: H)$ which is exponential in $k$ in the case that $k\geq3$.
Let
$$r_k=\begin{cases}
r_{*}(H)-1,  & \text{if $k=2$,}\\
(r_{*}(\mathscr{H})-1)\cdot(R_2(H)-1)\cdot (m(H)-1)^{(k-4)/2},  & \text{if $k\geq4$ is even,}\\
(\chi(H)-2)\cdot (R_2(H)-1)\cdot (m(H)-1)^{(k - 3)/2},  & \text{if $k$ is odd,}
\end{cases}$$
and $N_{n_{k}-1}^{k}$ be the graph defined in Definition~\ref{Def:nonbipartite}, where $m(H)=R_{2}(\mathscr{H})$ and
$$n_k=\begin{cases}
(R_2(H)-1)\cdot (m(H)-1)^{(k-2)/2}+1,  & \text{if $k$ is even,}\\
(\chi(H)-1)\cdot (R_2(H)-1)\cdot (m(H)-1)^{(k - 3)/2}+1,  & \text{if $k$ is odd.}
\end{cases}$$
Now we construct a $k$-edge-colored $N_{n_{k}-1}^{k}\sqcup K_{1, r_{k}}$ containing neither a rainbow triangle nor a monochromatic $H$ by induction on $k$.
Let $v$ be the center vertex of $K_{1, r_k}$.
When $k=2$, we have that $n_2=R_{2}(H)$ and $r_2=r_{*}(H)-1$. Then there exists a 2-edge-colored critical graph $N_{n_{2}-1}^{2}\sqcup K_{1, r_2}$ containing no monochromatic $H$.
Suppose that $k\geq 3$ and for any $2i<k$, we have constructed $N_{n_{2i}-1}^{2i}\sqcup K_{1, r_{2i}}$ which contains neither a rainbow triangle nor a monochromatic $H$.
When $2i+2\leq k$, $N_{n_{2i+2}-1}^{2i+2}$ is a blow-up of $D$ on $N_{n_{2i}-1}^{2i}$ by Definition~\ref{Def:nonbipartite}.
Denote each copy of $N_{n_{2i}-1}^{2i}$ in $N_{n_{2i+2}-1}^{2i+2}$ by $G_1$, $\ldots$, $G_{m(H)-1}$.
Let $v_j\in V(G_j)$ for each $1\leq j\leq m(H)-1$.
Then the subgraph induced by $\{v_1, \ldots, v_{m(H)-1}\}$ is isomorphic to $D$.
Choose a critical 2-edge-coloring $g$ of $D\sqcup K_{1, r_{*}(\mathscr{H})-1}$ (using $2i+1$ and $2i+2$) such that $D\sqcup K_{1, r_{*}(\mathscr{H})-1}$ contains no monochromatic graph $H^{'}$ in $\mathscr{H}$.
Let $v^{'}$ be the center vertex of $K_{1, r_{*}(\mathscr{H})-1}$.
Now we construct $N_{n_{2i+2}-1}^{2i+2}\sqcup K_{1, r_{2i+2}}$ according to the graph $D\sqcup K_{1, r_{*}(\mathscr{H})-1}$.
For any $v_j\in V(D)$ with $1\leq j\leq m(H)-1$, if $v^{'}v_j\in E(D\sqcup K_{1, r_{*}(\mathscr{H})-1})$, then we add the edge set $\{vu~|~u\in V(G_j)\}$.
The resulting graph is $N_{n_{2i+2}-1}^{2i+2}\sqcup K_{1, r_{2i+2}}$.
Thus, $r_{2i+2}=(r_{*}(\mathscr{H})-1)\cdot(n_{2i}-1)$.
Let $c$ be a $(2i+2)$-edge-coloring of $N_{n_{2i+2}-1}^{2i+2}\sqcup K_{1, r_{2i+2}}$ such that
$$c(e)=\begin{cases}
f(e),  & \text{if $e\in E(N_{n_{2i+2}-1}^{2i+2})$, where $f$ is defined in Definition~\ref{Def:nonbipartite},}\\
g(v^{'}v_j),  & \text{if $e\in E(v, V(G_j))$ for $1\leq j\leq m(H)-1$.}
\end{cases}$$
Clearly, this $(2i+2)$-edge-colored $N_{n_{2i+2}-1}^{2i+2}\sqcup K_{1, r_{2i+2}}$ by $c$ contains no rainbow triangle, where  $d(v)=r_{2i+2}=(r_{*}(\mathscr{H})-1)\cdot(R_2(H)-1)\cdot (m(H)-1)^{(2i-2)/2}$.
Suppose, to the contrary, that $N_{n_{2i+2}-1}^{2i+2}\sqcup K_{1, r_{2i+2}}$ contains a monochromatic $H$.
Then the monochromatic $H$ is colored by $2i+1$ or $2i+2$.
Hence for each copy of $N_{n_{2i}-1}^{2i}$, $V(H)\cap V(N_{n_{2i}-1}^{2i})$ is an independent set of $H$.
It follows that $D\sqcup K_{1, r_{*}(\mathscr{H})-1}$ contains a monochromatic merge graph of $H$, a contradiction.
When $2i+1=k$, $N_{n_{2i+1}-1}^{2i+1}$ is a blow-up of $K_{\chi(H)-1}$ on $N_{n_{2i}-1}^{2i}$ by Definition~\ref{Def:nonbipartite}.
Denote each copy of $N_{n_{2i}-1}^{2i}$ in $N_{n_{2i+1}-1}^{2i+1}$ by $G_1$, $\ldots$, $G_{\chi(H)-1}$.
We construct $N_{n_{2i+1}-1}^{2i+1}\sqcup K_{1, r_{2i+1}}$ by adding the edge set $\{vu~|~u\in \bigcup_{j=1}^{\chi(H)-2}V(G_j)\}$ to $N_{n_{2i+1}-1}^{2i+1}$.
Thus, $r_{2i+1}=(\chi(H)-2)\cdot(n_{2i}-1)$.
Let $c$ be a $(2i+1)$-edge-coloring of $N_{n_{2i+1}-1}^{2i+1}\sqcup K_{1, r_{2i+1}}$ such that
$$c(e)=\begin{cases}
f(e),  & \text{if $e\in E(N_{n_{2i+1}-1}^{2i+1})$, where $f$ is defined in Definition~\ref{Def:nonbipartite},}\\
2i+1,  & \text{if $e\in E(v, V(G_j))$ for $1\leq j\leq \chi(H)-2$.}
\end{cases}$$
Clearly, this is a $(2i+1)$-edge-colored $N_{n_{2i+1}-1}^{2i+1}\sqcup K_{1, r_{2i+1}}$ containing neither a rainbow triangle nor a monochromatic $H$, where $d(v)=r_{2i+1}=(\chi(H)-2)\cdot (R_2(H)-1)\cdot (m(H)-1)^{(2i - 2)/2}$.
Hence, we have that
$gr_{k}^{*}(K_3: H)\geq r_{k}+1
$, where $r_{k}+1$ is exponential in $k$ for $k\geq 3$.
For the upper bound, by Lemma~\ref{Lem:graph}, we know that $gr_{k}(K_3: H)\leq (R_2(H)-1)^{k(|H|-1)+1}$. Clearly, $gr_{k}^{*}(K_{3}: H)\leq (R_2(H)-1)^{k(|H|-1)+1}-1$. Thus, for non-bipartite graph $H$, there is a lower bound and an upper bound exponential in $k$.
Complete the proof of Theorem~\ref{Thm:graph}.
\end{proof}

\bibliography{bibfile1}

\end{document}